\documentclass[oneside,reqno,english]{amsart}
\usepackage[T1]{fontenc}
\usepackage[utf8]{inputenc}
\usepackage{xcolor}
\usepackage{babel}
\usepackage{prettyref}
\usepackage{amstext}
\usepackage{amsthm}
\usepackage{amssymb}
\usepackage[pdfusetitle,
 bookmarks=true,bookmarksnumbered=false,bookmarksopen=false,
 breaklinks=false,pdfborder={0 0 0},pdfborderstyle={},backref=false,colorlinks=false]
 {hyperref}
\hypersetup{
 colorlinks=true,citecolor=blue,linkcolor=blue,linktocpage=true}

\makeatletter
\numberwithin{equation}{section}
\numberwithin{figure}{section}

\usepackage{prettyref}

\newrefformat{cor}{Corollary~\ref{#1}}
\newrefformat{subsec}{Section~\ref{#1}}
\newrefformat{lem}{Lemma~\ref{#1}}
\newrefformat{thm}{Theorem~\ref{#1}}
\newrefformat{sec}{Section~\ref{#1}}
\newrefformat{chap}{Chapter~\ref{#1}}
\newrefformat{prop}{Proposition~\ref{#1}}
\newrefformat{exa}{Example~\ref{#1}}
\newrefformat{tab}{Table~\ref{#1}}
\newrefformat{rem}{Remark~\ref{#1}}
\newrefformat{def}{Definition~\ref{#1}}
\newrefformat{fig}{Figure~\ref{#1}}
\newrefformat{claim}{Claim~\ref{#1}}
\newrefformat{assu}{Assumption~\ref{#1}}

\makeatother

\theoremstyle{plain}
\newtheorem{thm}{\protect\theoremname}[section]
\theoremstyle{definition}
\newtheorem{defn}[thm]{\protect\definitionname}
\theoremstyle{plain}
\newtheorem{lem}[thm]{\protect\lemmaname}
\newtheorem{cor}[thm]{\protect\corollaryname}
\newtheorem{prop}[thm]{\protect\propositionname}
\theoremstyle{remark}
\newtheorem{rem}[thm]{\protect\remarkname}
\newtheorem*{rem*}{\protect\remarkname}
\theoremstyle{definition}
\newtheorem{example}[thm]{\protect\examplename}
\providecommand{\corollaryname}{Corollary}
\providecommand{\definitionname}{Definition}
\providecommand{\examplename}{Example}
\providecommand{\lemmaname}{Lemma}
\providecommand{\propositionname}{Proposition}
\providecommand{\remarkname}{Remark}
\providecommand{\theoremname}{Theorem}

\begin{document}
\title[]{Compression Covariance and Tangent kernels}
\begin{abstract}
Let $A\geq0$ be self-adjoint on a Hilbert space $H$, let $T_{t}=e^{-tA}$,
and let $P$ be an orthogonal projection. Relative to the decomposition
$H=PH\oplus P^{\perp}H$, write 
\[
T_{t}=\begin{pmatrix}C_{t} & V^{*}_{t}\\
V_{t} & D_{t}
\end{pmatrix},
\]
where $C_{t}=PT_{t}P|_{PH}$, $V_{t}=P^{\perp}T_{t}P|_{PH}$, and
$D_{t}=P^{\perp}T_{t}P^{\perp}|_{P^{\perp}H}$. The compressed family
$\left(C_{t}\right)$ consists of positive contractions but need not
form a semigroup. Its defect is given by 
\[
C_{s+t}-C_{s}C_{t}=V^{*}_{s}V_{t}
\]
while the complementary block satisfies 
\[
D_{s+t}-D_{s}D_{t}=V_{s}V^{*}_{t}.
\]
Thus the failure of $\left\{ C_{t}\right\} $ and $\left\{ D_{t}\right\} $
to be semigroups gives two Gram kernels associated with the same off-diagonal
maps. We treat these covariance defects as positive definite operator-valued
kernels and use their Kolmogorov spaces to recover the hidden dynamics
they encode. We then study short-time rescalings of $E_{s,t}:=V^{*}_{s}V_{t}$.
The tangent kernel 
\[
F\left(s,t\right):=\lim_{\varepsilon\downarrow0}a\left(\varepsilon\right)^{-1}E_{\varepsilon s,\varepsilon t}
\]
has its own Kolmogorov space, and the lower-right block dynamics induces
a positive self-adjoint contraction semigroup on it. The representing
vectors of $F$ then satisfy an additive cocycle identity for this
semigroup. This gives an intrinsic restriction on the positive kernels
that can arise as short-time compression covariance tangents.
\end{abstract}

\author{James Tian}
\address{(James Tian) Mathematical Reviews, 535 W. William St, Suite 210, Ann
Arbor, MI 48103, USA}
\email{james.ftian@gmail.com}

\maketitle
\tableofcontents{}

\section{Introduction}\label{sec:1}

The compressed family $\left(C_{t}\right)_{t\geq0}$ is the evolution
seen from the observed subspace $PH$. It consists of positive contractions,
but compression usually destroys the semigroup law. The quantity 
\[
E_{s,t}=C_{s+t}-C_{s}C_{t}
\]
measures the part of the full evolution which is lost from the observed
subspace during one time interval and then contributes again after
further evolution. The block identity 
\[
E_{s,t}=V^{*}_{s}V_{t}
\]
shows that this loss has a positive covariance form. It is the Gram
kernel of the hidden components $V_{t}h$ produced by the compressed
evolution.

The same off-diagonal maps also produce a covariance on the hidden
side. The lower-right block satisfies 
\[
D_{s+t}-D_{s}D_{t}=V_{s}V^{*}_{t}.
\]
Thus the visible loss of semigroup structure and the hidden covariance
are paired from the beginning. They are the two Gram pictures associated
with the same maps from the observed subspace into $P^{\perp}H$.
This is the basic reason for treating $E$ as a covariance kernel,
not just as a remainder term.

\prettyref{sec:2} develops the elementary kernel structure of this
object. The identity $E_{s,t}=V^{*}_{s}V_{t}$ gives positivity, while
the semigroup law gives a covariance identity relating $E_{r+s,t}$,
$E_{s,t}$, and $E_{r,s+t}$. We also include a converse dilation
statement. Starting from a strongly continuous self-adjoint family
$\left(C_{t}\right)$, positivity of the Hankel kernel $\left(s,t\right)\mapsto C_{s+t}$
builds the dilation space, while positivity of the shifted difference
\[
\left(s,t\right)\mapsto C_{s+t}-C_{s+t+2r}
\]
makes the time-shift contractive. This gives a self-adjoint contraction
semigroup whose compression is the given family.

\prettyref{sec:3} studies the two-sided block dynamics. The space
generated by the leakage maps, 
\[
\overline{span}\left\{ V_{t}h:t\geq0,\ h\in PH\right\} ,
\]
reduces the lower-right block dynamics. Under the Kolmogorov identification
for $E$, the restriction of $D_{s}$ to this generated space is determined
by $C$ and $E$. In a minimal representation, the whole hidden summand
is recovered this way. Thus the compression covariance determines
the hidden space seen by the observed dynamics.

The same section also treats short-time behavior for bounded generators.
If 
\[
A=\begin{pmatrix}B & \Gamma^{*}\\
\Gamma & L
\end{pmatrix},
\]
then the first quadratic coefficient of the visible covariance is
$\Gamma^{*}\Gamma$, and the corresponding coefficient on the hidden
side is $\Gamma\Gamma^{*}$. For a finite block decomposition with
prescribed block support, this leading-order calculation refines into
a distance expansion. A hidden block at graph distance $r$ from the
observed block can first appear at order $s^{r}t^{r}$, unless the
corresponding coefficient vanishes. 

The path case includes finite Jacobi matrices, where the order at
which each site appears is governed by its distance from the observed
site. Finally, the same path structure has an intrinsic form: starting
from $PH$, the increasing Krylov subspaces 
\[
\mathcal{K}_{n}=\overline{span}\left\{ A^{j}PH:0\leq j\leq n\right\} 
\]
produce orthogonal layers on which $A$ is block tridiagonal. Thus
the distance expansion is not only a feature of a preassigned graph;
it is also the canonical block structure generated by the observed
subspace itself.

\prettyref{sec:4} passes from finite-order expansions to short-time
tangent kernels. For a scale function $a\left(\varepsilon\right)$,
we study limits 
\[
a\left(\varepsilon\right)^{-1}E_{\varepsilon s,\varepsilon t}\to F\left(s,t\right).
\]
The main point is that the lower-right block dynamics survives this
rescaling. On the Kolmogorov space of $F$, the formula 
\[
S_{r}w_{t}h=w_{r+t}h-w_{r}h
\]
defines a positive self-adjoint contraction semigroup. Hence leakage
tangents are cocycle kernels in an intrinsic sense. This gives kernel
conditions beyond positivity, including the shifted-increment inequalities
$0\leq F^{[r]}\leq F$ and the corresponding transport identity. We
also describe spectral cocycle models, including fractional kernels
of the form 
\[
s^{\beta}+t^{\beta}-\left(s+t\right)^{\beta},\qquad0<\beta<1,
\]
arising from singular spectral measures.

\prettyref{sec:5} returns to the hidden covariance at the tangent
level. The lower-right covariance usually does not have a direct operator-kernel
limit on the original hidden space. After testing against vectors
measured at the tangent scale, however, it has a natural limit on
the Kolmogorov space of $F$. This produces an intrinsic hidden-side
tangent covariance on the same space that carries the semigroup $\left(S_{r}\right)$.
The resulting identities show that hidden increments are transported
by $S_{r}$, and that the visible shifted-increment kernels form a
decreasing hierarchy.

\subsection*{Literature context}

The construction is close to several familiar themes, although the
emphasis here is different. In dilation theory and semigroup theory,
compressed evolutions are often studied through the larger space on
which the original semigroup acts. The covariance kernel considered
here keeps track of the part of this dilation which is lost under
compression and later returns to the observed subspace. In this sense
it gives a positive kernel form of the non-semigroup part of the compressed
dynamics. Related objects also occur in system theory and product
system theory, where off-diagonal blocks, covariance data, and cocycles
encode input, output, and transfer information \cite{MR987590,MR1047579,MR2065240}.
Here the same block entries are used to produce operator-valued covariance
kernels and their tangent limits.

There is also a parallel with stochastic analysis. Covariance kernels
of Gaussian processes are positive kernels, and increments or conditional
increments often carry additional structure beyond positivity. The
kernel $E_{s,t}=V^{*}_{s}V_{t}$ may be read as the covariance of
the hidden part of the evolution seen from the observed subspace.
Under short-time rescaling, the limiting kernel is not an arbitrary
covariance kernel. It carries a semigroup on its kernel space, so
the limiting covariance has the form of an additive cocycle. This
is analogous to situations where covariance structures retain a Markovian,
orthogonal-dynamics, or cocycle-type structure, even though the object
is described through its covariance \cite{MR343815,MR343816,MR2154727,PhysRevE.99.062118,MR4396389,MR4966227}.

A related point of view comes from partial differential equations
and Dirichlet form theory. Compressing a semigroup to a subspace is
a Hilbert space analogue of observing an evolution only on part of
its state space or through a restricted class of data. The difference
between the compressed evolution and a true semigroup encodes an exchange
with hidden variables. Short-time asymptotics of this exchange are
sensitive to how the generator couples the observed component to the
unobserved one. In finite block models this coupling is reflected
in graph distance. In more singular situations one expects non-integer
scales, and the tangent formalism isolates the structure which any
such scale limit must satisfy. Boundary semigroups, Dirichlet-to-Neumann
operators, singular diffusions, and boundary representations of forms
give related settings in which observed and hidden components interact
through a semigroup or form-theoretic construction \cite{MR2823661,MR2912743,MR4166804,MR4523575,MR4907186,MR4116717,MR4997328,MR2906858}.

The positive kernel formulation is also related to reproducing kernel
Hilbert spaces and Kolmogorov decompositions of operator-valued kernels.
The kernel space associated with $E$ identifies the hidden space
generated by the leakage maps, while the kernel space associated with
a tangent $F$ carries the limiting hidden semigroup. Thus the paper
uses positive kernel methods not only to represent a covariance, but
also to recover the residual dynamics which remains after compression
and short-time rescaling \cite{MR2265340,MR3027907,MR3620761}.

Fractional kernels also appear among the model cocycle kernels. In
\prettyref{sec:4}, kernels of the form $s^{\beta}+t^{\beta}-\left(s+t\right)^{\beta}$,
with $0<\beta<1$, are obtained from singular spectral measures with
density proportional to $x^{-1-\beta}dx$. Thus fractional powers
and singular spectral measures provide natural examples of the tangent
objects studied here. This is related to the literature on nonlocal
Dirichlet forms, fractional Laplacians, fractional boundary operators,
and generalized fractional derivatives \cite{MR2465826,MR2677618,MR3050510,MR3168912,MR3318251,MR3613319,MR3920522,MR4585119,MR4779596}.
The present paper does not attempt to realize all such spectral cocycles
as tangents of prescribed pre-limit generators.

\section{Compression covariances}\label{sec:2}

We begin with the basic kernel attached to a compressed self-adjoint
semigroup. The compression itself is a family of positive contractions
and need not satisfy the semigroup law. The failure of the semigroup
law, however, has a canonical positive form: it is the Gram kernel
of the off-diagonal part of the original semigroup. This section gives
the elementary identities behind this fact and an intrinsic dilation
criterion for those families that arise as compressions.

Let $H$ be a Hilbert space, let $A\geq0$ be self-adjoint on $H$,
and let 
\[
T_{t}=e^{-tA},\qquad t\geq0,
\]
be the associated self-adjoint contraction semigroup. Let $P$ be
an orthogonal projection on $H$, and put 
\[
K=PH,\qquad P^{\perp}=I-P.
\]
Define 
\[
C_{t}=PT_{t}P|_{K},\qquad t\geq0.
\]
Thus, for $h\in K$, 
\[
C_{t}h=PT_{t}h.
\]

\begin{defn}
\label{def:2-1} For $s,t\geq0$, define 
\begin{equation}
E_{s,t}:=C_{s+t}-C_{s}C_{t}\in B\left(K\right).\label{eq:2-1}
\end{equation}
We call $E=\left(E_{s,t}\right)_{s,t\geq0}$ the compression covariance
kernel associated with $\left(A,P\right)$. 
\end{defn}

The identity \prettyref{eq:2-2} below gives a basic interpretation
of $E_{s,t}$. The map $V_{t}:=P^{\perp}T_{t}P|_{K}$ is the part
of the evolved vector that has left $K$. Thus $E$ is the Gram kernel
of the leakage maps, i.e., 
\[
E_{s,t}=V^{*}_{s}V_{t}.
\]
It measures the part of $T_{t}h$ that is outside $K$ after the first
time interval, but which later contributes back to the final observation
in $K$ after evolving for time $s$.
\begin{lem}
\label{lem:2-2} For all $s,t\geq0$, 
\begin{equation}
E_{s,t}=PT_{s}P^{\perp}T_{t}P|_{K}.\label{eq:2-2}
\end{equation}
The kernel $E$ is positive as a $B\left(K\right)$-valued kernel.
For all $r,s,t\geq0$, 
\begin{equation}
E_{r+s,t}=C_{r}E_{s,t}+E_{r,s+t}-E_{r,s}C_{t}.\label{eq:2-3}
\end{equation}
Also, 
\begin{equation}
E_{s,t}=\left(I_{K}-C_{s}\right)+\left(I_{K}-C_{t}\right)-\left(I_{K}-C_{s+t}\right)-\left(I_{K}-C_{s}\right)\left(I_{K}-C_{t}\right).\label{eq:2-4}
\end{equation}
\end{lem}

\begin{proof}
Let $h\in K$. Since $Ph=h$, 
\[
C_{s+t}h=PT_{s+t}h=PT_{s}T_{t}h
\]
and 
\[
C_{s}C_{t}h=C_{s}\left(PT_{t}h\right)=PT_{s}PT_{t}h.
\]
Subtracting gives 
\[
E_{s,t}h=PT_{s}\left(I-P\right)T_{t}h=PT_{s}P^{\perp}T_{t}h,
\]
which proves \prettyref{eq:2-2}.

For $t\geq0$, define 
\[
V_{t}:K\to P^{\perp}H,\qquad V_{t}h=P^{\perp}T_{t}h.
\]
Then $V_{t}=P^{\perp}T_{t}P$ as an operator on $H$. Since $T_{t}=T^{*}_{t}$,
\[
V^{*}_{s}=\left(P^{\perp}T_{s}P\right)^{*}=PT_{s}P^{\perp},
\]
and hence 
\[
E_{s,t}=V^{*}_{s}V_{t}.
\]
Therefore, for $s_{1},\ldots,s_{n}\geq0$ and $h_{1},\ldots,h_{n}\in K$,
\[
\sum^{n}_{i,j=1}\left\langle h_{i},E_{s_{i},s_{j}}h_{j}\right\rangle =\Big\Vert\sum^{n}_{j=1}V_{s_{j}}h_{j}\Big\Vert^{2}\geq0.
\]
Thus $E$ is positive.

By \prettyref{eq:2-1}, 
\[
C_{a+b}=C_{a}C_{b}+E_{a,b}.
\]
Applying this identity to $C_{r+s+t}$ with the groupings $\left(r+s\right)+t$
and $r+\left(s+t\right)$ gives 
\[
\left(C_{r}C_{s}+E_{r,s}\right)C_{t}+E_{r+s,t}=C_{r}\left(C_{s}C_{t}+E_{s,t}\right)+E_{r,s+t}.
\]
After cancelling $C_{r}C_{s}C_{t}$, we obtain \prettyref{eq:2-3}.

Finally, set $L_{t}=I_{K}-C_{t}$. Since $C_{t}=I_{K}-L_{t}$, 
\[
E_{s,t}=\left(I_{K}-L_{s+t}\right)-\left(I_{K}-L_{s}\right)\left(I_{K}-L_{t}\right)=L_{s}+L_{t}-L_{s+t}-L_{s}L_{t},
\]
which is \prettyref{eq:2-4}. 
\end{proof}

The ideas in the next result are classical; see e.g., \cite{MR51437,MR2760647,MR2938971},
and the recent survey \cite{MR4250453}. The relevant intrinsic object
here is the Hankel kernel \prettyref{eq:2-5-0}. Its positivity gives
the Kolmogorov space, while the positivity of \prettyref{eq:2-5-1}
is the condition that the shift $\left[t,h\right]\mapsto\left[t+r,h\right]$
is contractive in that space.
\begin{thm}
\label{thm:2-3} Let $K$ be a Hilbert space, and let $\left(C_{t}\right)_{t\geq0}\subset B\left(K\right)$
be a strongly continuous family such that $C_{0}=I_{K}$ and $C_{t}=C^{*}_{t}$
for all $t\geq0$. Assume that the kernel 
\begin{equation}
\left(s,t\right)\mapsto C_{s+t}\label{eq:2-5-0}
\end{equation}
is positive as a $B\left(K\right)$-valued kernel. Assume also that,
for every $r\geq0$, the kernel 
\begin{equation}
\left(s,t\right)\mapsto C_{s+t}-C_{s+t+2r}\label{eq:2-5-1}
\end{equation}
is positive as a $B\left(K\right)$-valued kernel.

Then there exist a Hilbert space $H$, an isometry $W:K\to H$, and
a strongly continuous self-adjoint contraction semigroup $\left(T_{t}\right)_{t\geq0}$
on $H$ such that 
\begin{equation}
C_{t}=W^{*}T_{t}W\label{eq:2-5}
\end{equation}
for all $t\geq0$. After identifying $K$ with $WK$ and writing $P$
for the orthogonal projection of $H$ onto $K$, 
\begin{equation}
C_{t}=PT_{t}P|_{K}\label{eq:2-6}
\end{equation}
for all $t\geq0$. The representation may be chosen minimal in the
sense that 
\begin{equation}
H=\overline{span}\left\{ T_{t}Wh:t\geq0,\ h\in K\right\} .\label{eq:2-7}
\end{equation}
For this representation, 
\begin{equation}
C_{s+t}-C_{s}C_{t}=PT_{s}P^{\perp}T_{t}P|_{K}\label{eq:2-8}
\end{equation}
for all $s,t\geq0$. 
\end{thm}

\begin{proof}
Let $D_{0}$ be the algebraic span of symbols $\left[t,h\right]$,
where $t\geq0$ and $h\in K$. Define a sesquilinear form on $D_{0}$
by 
\[
\left\langle \sum^{n}_{i=1}\left[s_{i},h_{i}\right],\sum^{m}_{j=1}\left[t_{j},k_{j}\right]\right\rangle =\sum^{n}_{i=1}\sum^{m}_{j=1}\left\langle h_{i},C_{s_{i}+t_{j}}k_{j}\right\rangle .
\]
The positivity of \prettyref{eq:2-5-0} implies that this form is
positive semidefinite. Quotient by its null space and complete. Let
$H$ be the resulting Hilbert space. We denote the class of $\left[t,h\right]$
again by $\left[t,h\right]$.

Define $W:K\to H$ by $Wh=\left[0,h\right]$. Since $C_{0}=I_{K}$,
\[
\left\Vert Wh\right\Vert ^{2}=\left\langle h,C_{0}h\right\rangle =\left\Vert h\right\Vert ^{2}.
\]
Thus $W$ is an isometry.

For $r\geq0$, define on $D_{0}$ the map 
\[
T_{r}\left[t,h\right]=\left[t+r,h\right].
\]
Let 
\[
x=\sum^{n}_{i=1}\left[s_{i},h_{i}\right].
\]
Then 
\[
\left\Vert x\right\Vert ^{2}-\left\Vert T_{r}x\right\Vert ^{2}=\sum^{n}_{i,j=1}\left\langle h_{i},\left(C_{s_{i}+s_{j}}-C_{s_{i}+s_{j}+2r}\right)h_{j}\right\rangle .
\]
This is nonnegative by the assumed positivity \prettyref{eq:2-5-1}.
Hence $T_{r}$ is contractive on $D_{0}$. If $x$ has zero norm,
then 
\[
0\leq\left\Vert T_{r}x\right\Vert ^{2}\leq\left\Vert x\right\Vert ^{2}=0.
\]
So $T_{r}$ descends to the quotient and extends to a contraction
on $H$.

The operator $T_{r}$ is self-adjoint. On generators, 
\[
\left\langle T_{r}\left[s,h\right],\left[t,k\right]\right\rangle =\left\langle h,C_{s+r+t}k\right\rangle =\left\langle \left[s,h\right],T_{r}\left[t,k\right]\right\rangle .
\]
The semigroup law follows on generators: 
\[
T_{r}T_{q}\left[t,h\right]=\left[t+q+r,h\right]=T_{r+q}\left[t,h\right].
\]
Thus $\left(T_{t}\right)_{t\geq0}$ is a semigroup of self-adjoint
contractions.

We check strong continuity on the dense span of the vectors $\left[t,h\right]$.
For $r,q\geq0$, 
\[
\begin{aligned}\left\Vert T_{r}\left[t,h\right]-T_{q}\left[t,h\right]\right\Vert ^{2} & =\left\langle h,C_{2t+2r}h\right\rangle +\left\langle h,C_{2t+2q}h\right\rangle \\
 & \quad-\left\langle h,C_{2t+r+q}h\right\rangle -\left\langle h,C_{2t+q+r}h\right\rangle .
\end{aligned}
\]
The right side tends to $0$ as $r\to q$, since $u\mapsto C_{u}$
is strongly continuous. Since each $T_{r}$ is a contraction, strong
continuity follows on all of $H$.

For $h,k\in K$, 
\[
\left\langle h,W^{*}T_{t}Wk\right\rangle =\left\langle Wh,T_{t}Wk\right\rangle =\left\langle \left[0,h\right],\left[t,k\right]\right\rangle =\left\langle h,C_{t}k\right\rangle .
\]
Hence \prettyref{eq:2-5} holds. The form \prettyref{eq:2-6} follows
after identifying $K$ with $WK$. Minimality follows from the construction,
since $\left[t,h\right]=T_{t}Wh$.

Since $T_{t}=T^{2}_{t/2}$ and $T_{t/2}$ is self-adjoint, $T_{t}\geq0$
for $t\geq0$. Hence $\left(T_{t}\right)_{t\geq0}$ is a positive
self-adjoint contraction semigroup.

It remains to prove \prettyref{eq:2-8}. By \prettyref{eq:2-6}, 
\[
C_{s+t}=PT_{s+t}P|_{K}=PT_{s}T_{t}P|_{K}
\]
and 
\[
C_{s}C_{t}=PT_{s}PT_{t}P|_{K}.
\]
Subtracting gives 
\[
C_{s+t}-C_{s}C_{t}=PT_{s}\left(I-P\right)T_{t}P|_{K}=PT_{s}P^{\perp}T_{t}P|_{K}.
\]
This proves \prettyref{eq:2-8}. 
\end{proof}

\section{Block dynamics}\label{sec:3}

We now look more closely at the block structure behind the covariance
kernel. The compression covariance 
\[
E_{s,t}=V^{*}_{s}V_{t}
\]
is only one side of the off-diagonal dynamics. The same maps $V_{t}$
also give the hidden-side covariance 
\[
D_{s+t}-D_{s}D_{t}=V_{s}V^{*}_{t}.
\]
The first part of this section states the block identities which connect
these two kernels and shows that the compression covariance recovers
the hidden space generated by the observed dynamics.

The second part studies what these identities say at short time when
the generator is bounded. In that case the off-diagonal block of $A$
gives the first quadratic coefficient, and a finer block decomposition
gives distance-based higher-order terms. The path case leads to finite
Jacobi models, and then to an intrinsic Krylov construction in which
the subspaces generated from $PH$ produce a block tridiagonal form
without choosing a graph in advance.

\subsection{Two-sided block dynamics}

We now write the semigroup in blocks relative to the decomposition
$H=PH\oplus P^{\perp}H$. Thus 
\[
T_{t}=\begin{pmatrix}C_{t} & V^{*}_{t}\\
V_{t} & D_{t}
\end{pmatrix},
\]
where $V_{t}=P^{\perp}T_{t}P|_{PH}$ and $D_{t}=P^{\perp}T_{t}P^{\perp}|_{P^{\perp}H}$.
The semigroup identity $T_{s+t}=T_{s}T_{t}$ gives the basic relations
among these blocks. They show that the compression covariance $E_{s,t}=V^{*}_{s}V_{t}$
on $PH$ is coupled to a hidden covariance $V_{s}V^{*}_{t}$ on $P^{\perp}H$.
\begin{lem}
\label{lem:3-1} For all $s,t\geq0$, 
\begin{align}
C_{s+t} & =C_{s}C_{t}+V^{*}_{s}V_{t},\label{eq:3-1}\\
V_{s+t} & =V_{s}C_{t}+D_{s}V_{t},\label{eq:3-2}\\
D_{s+t} & =D_{s}D_{t}+V_{s}V^{*}_{t}.\label{eq:3-3}
\end{align}
\end{lem}

\begin{proof}
The identities are the block entries of 
\[
T_{s+t}=T_{s}T_{t}
\]
relative to the decomposition $H=PH\oplus P^{\perp}H$: 
\[
\begin{pmatrix}C_{s+t} & V^{*}_{s+t}\\
V_{s+t} & D_{s+t}
\end{pmatrix}=\begin{pmatrix}C_{s} & V^{*}_{s}\\
V_{s} & D_{s}
\end{pmatrix}\begin{pmatrix}C_{t} & V^{*}_{t}\\
V_{t} & D_{t}
\end{pmatrix}.
\]
The upper-left entry gives \prettyref{eq:3-1}. Since $E_{s,t}:=C_{s+t}-C_{s}C_{t}$,
this also gives $E_{s,t}=V^{*}_{s}V_{t}$. The lower-left and lower-right
entries give \prettyref{eq:3-2} and \prettyref{eq:3-3}. 
\end{proof}

Thus the same off-diagonal maps $V_{t}$ generate two covariance kernels.
The first is the compression covariance 
\[
E_{s,t}=V^{*}_{s}V_{t}
\]
on $PH$. The second is the lower-right covariance 
\[
D_{s+t}-D_{s}D_{t}=V_{s}V^{*}_{t}
\]
on $P^{\perp}H$. These are the same off-diagonal terms viewed from
the two sides of the decomposition $H=PH\oplus P^{\perp}H$.

The space generated by the leakage maps is 
\[
\mathcal{L}_{E}=\overline{span}\left\{ V_{t}h:t\geq0,\ h\in PH\right\} \subseteq P^{\perp}H.
\]
This space is determined by $E$, up to the usual unitary identification
of Kolmogorov spaces.
\begin{lem}
\label{lem:3-2} For every $s\geq0$, the space $\mathcal{L}_{E}$
reduces $D_{s}$. Moreover, for all $s,t\geq0$ and $h\in PH$, 
\begin{equation}
D_{s}V_{t}h=V_{s+t}h-V_{s}C_{t}h.\label{eq:3-4}
\end{equation}
Consequently, after identifying $\mathcal{L}_{E}$ with the Kolmogorov
space of $E$, the restriction of $D_{s}$ to $\mathcal{L}_{E}$ is
determined by $C$ and $E$.
\end{lem}

\begin{proof}
The identity \prettyref{eq:3-4} is just \prettyref{eq:3-2} rewritten.
It implies that $D_{s}$ maps the linear span of 
\[
\left\{ V_{t}h:t\geq0,\ h\in PH\right\} 
\]
into itself. Hence $D_{s}\mathcal{L}_{E}\subseteq\mathcal{L}_{E}$.
Since $D_{s}$ is self-adjoint on $P^{\perp}H$, the closed invariant
subspace $\mathcal{L}_{E}$ reduces $D_{s}$.

To see the last assertion, let $H_{E}$ be the Kolmogorov space of
the positive kernel $E$. Thus $H_{E}$ is the completion of the span
of symbols $\left[t,h\right]$, with inner product 
\[
\left\langle \left[s,h\right],\left[t,k\right]\right\rangle =\left\langle h,E_{s,t}k\right\rangle .
\]
The map 
\[
\left[t,h\right]\mapsto V_{t}h
\]
extends to a unitary from $H_{E}$ onto $\mathcal{L}_{E}$, since
$E_{s,t}=V^{*}_{s}V_{t}$. Under this unitary, the restriction of
$D_{s}$ to $\mathcal{L}_{E}$ is given on generators by 
\[
\left[t,h\right]\mapsto\left[s+t,h\right]-\left[s,C_{t}h\right].
\]
Thus this restriction is determined by $C$ and $E$. 
\end{proof}

\begin{cor}
\label{cor:3-3} Assume that the representation $\left(H,T_{t},P\right)$
is minimal in the sense that $H=\overline{span}\left\{ T_{t}h:t\geq0,\:h\in PH\right\} $.
Then 
\[
P^{\perp}H=\mathcal{L}_{E}.
\]
Equivalently, in a minimal representation the hidden summand is the
Kolmogorov space of the compression covariance kernel. 
\end{cor}

\begin{proof}
If $y\in P^{\perp}H$ is orthogonal to $\mathcal{L}_{E}$, then 
\[
\left\langle y,V_{t}h\right\rangle =0
\]
for all $t\geq0$ and $h\in PH$. Since 
\[
T_{t}h=C_{t}h+V_{t}h,
\]
with $C_{t}h\in PH$ and $V_{t}h\in P^{\perp}H$, and since $y\in P^{\perp}H$,
we have 
\[
\left\langle y,C_{t}h\right\rangle =0.
\]
Therefore 
\[
\left\langle y,T_{t}h\right\rangle =\left\langle y,C_{t}h\right\rangle +\left\langle y,V_{t}h\right\rangle =0
\]
for all $t\geq0$ and $h\in PH$. By the minimality assumption, $y$
is orthogonal to all of $H$, so $y=0$. Thus 
\[
P^{\perp}H\ominus\mathcal{L}_{E}=\{0\},
\]
and therefore $P^{\perp}H=\mathcal{L}_{E}$. 
\end{proof}

\begin{cor}
\label{cor:3-4} The following are equivalent: 
\begin{enumerate}
\item $E_{s,t}=0$ for all $s,t\geq0$; 
\item $V_{t}=0$ for all $t\geq0$; 
\item $PH$ is invariant for $T_{t}$ for every $t\geq0$; 
\item $\left(C_{t}\right)_{t\geq0}$ is a semigroup. 
\end{enumerate}
If these conditions hold and the representation $\left(H,T_{t},P\right)$
is minimal, then $P^{\perp}H=\{0\}$. Equivalently, the minimal representation
has no hidden summand, so $H=PH$. 
\end{cor}

\begin{proof}
The equivalence follows from $E_{s,t}=V^{*}_{s}V_{t}$, since $E_{t,t}=0$
implies $V_{t}=0$, from the definition $V_{t}=P^{\perp}T_{t}P|_{K}$,
and from the identity $C_{s+t}=C_{s}C_{t}+V^{*}_{s}V_{t}$; the final
assertion follows from \prettyref{cor:3-3}.
\end{proof}

The identities above describe the hidden component globally, through
the off-diagonal maps $V_{t}$. When the generator has bounded blocks,
these relations can also be read infinitesimally. The first-order
term of $V_{t}$ is controlled by the off-diagonal block of the generator,
and therefore the first short-time coefficient of the two covariance
kernels is obtained by squaring this coupling from the two sides.
\begin{prop}
\label{prop:3-5}Let $T_{t}=e^{-tA}$ be as above. Assume in addition
that $A\in B\left(H\right)$. Relative to $H=PH\oplus P^{\perp}H$,
write 
\[
A=\begin{pmatrix}B & \Gamma^{*}\\
\Gamma & L
\end{pmatrix},
\]
where 
\[
B=PAP|_{PH},\qquad L=P^{\perp}AP^{\perp}|_{P^{\perp}H},\qquad\Gamma=P^{\perp}AP|_{PH}.
\]
Then, in operator norm, as $\left(s,t\right)\to\left(0,0\right)$
with $s,t>0$, 
\[
\frac{1}{st}\left(C_{s+t}-C_{s}C_{t}\right)\to\Gamma^{*}\Gamma\tag{3.5}
\]
and 
\[
\frac{1}{st}\left(D_{s+t}-D_{s}D_{t}\right)\to\Gamma\Gamma^{*}.\tag{3.6}
\]
\end{prop}

\begin{proof}
Since $A$ is bounded, 
\[
e^{-tA}=I-tA+O\left(t^{2}\right),\quad\text{as }t\downarrow0.
\]
Taking the lower-left block and using $P^{\perp}P=0$, we get 
\[
V_{t}=P^{\perp}e^{-tA}P|_{PH}=P^{\perp}\left(e^{-tA}-I\right)P|_{PH}=-tP^{\perp}AP|_{PH}+O\left(t^{2}\right).
\]
Thus $V_{t}=-t\Gamma+O\left(t^{2}\right)$, and hence $t^{-1}V_{t}\to-\Gamma$.

By \prettyref{lem:3-1}, $C_{s+t}-C_{s}C_{t}=V^{*}_{s}V_{t}$. Therefore
\[
\frac{1}{st}\left(C_{s+t}-C_{s}C_{t}\right)=\left(\frac{1}{s}V_{s}\right)^{*}\left(\frac{1}{t}V_{t}\right)\to\Gamma^{*}\Gamma.
\]
The same argument, using $D_{s+t}-D_{s}D_{t}=V_{s}V^{*}_{t}$, gives
\[
\frac{1}{st}\left(D_{s+t}-D_{s}D_{t}\right)=\left(\frac{1}{s}V_{s}\right)\left(\frac{1}{t}V_{t}\right)^{*}\to\Gamma\Gamma^{*}.
\]
\end{proof}

\begin{rem}
The boundedness of $A$ is used only to get $V_{t}=-t\Gamma+O\left(t^{2}\right)$
in operator norm. The same proof applies whenever $t^{-1}V_{t}\to-\Gamma$
in operator norm for some $\Gamma\in B\left(PH,P^{\perp}H\right)$.
If $PH$ is finite-dimensional, this follows from $PH\subset\mathcal{D}\left(A\right)$,
with $\Gamma=P^{\perp}AP|_{PH}$. 

For unbounded generators, this conclusion is not automatic. The expansion
$e^{-tA}=I-tA+O\left(t^{2}\right)$ does not hold in operator norm,
and the block $P^{\perp}AP|_{PH}$ need not be defined on all of $PH$.
Even when $PH\subset\mathcal{D}\left(A\right)$, one generally obtains
only pointwise convergence of $t^{-1}P^{\perp}\left(e^{-tA}-I\right)P|_{PH}$
on $PH$, not norm convergence. 
\end{rem}

Thus, in the bounded-generator case, the two covariance kernels have
infinitesimal coefficients determined by the off-diagonal block $\Gamma$.
The upper-left coefficient is $\Gamma^{*}\Gamma$, and the lower-right
coefficient is $\Gamma\Gamma^{*}$. In \prettyref{sec:4}, we study
short-time limits in which this quadratic scaling is replaced by another
scale.

\subsection{Higher order dynamics}

We next give a higher-order form of the infinitesimal calculation
above. Let 
\[
H=PH\oplus P^{\perp}H=\underset{PH}{\underbrace{H_{1}}}\oplus\underset{P^{\perp}H}{\underbrace{\cdots\oplus H_{m}}}
\]
be a finite orthogonal decomposition, and let $P_{i}$ be the orthogonal
projection onto $H_{i}$. We take $H_{1}=PH$, $P_{1}=P$. 

Given $T_{t}=e^{-tA}$ with $A\in B\left(H\right)$, write 
\[
A_{ij}=P_{i}AP_{j}|_{H_{j}}.
\]

Let $G$ be a graph on $\left\{ 1,\ldots,m\right\} $. We say that
$A$ has block support in $G$ if 
\[
A_{ij}=0
\]
whenever $i\neq j$ and $\left\{ i,j\right\} $ is not an edge of
$G$.

For $k\geq2$, put 
\[
V^{\left(k\right)}_{t}=P_{k}e^{-tA}P_{1}|_{H_{1}}.
\]
Under the decomposition 
\[
P^{\perp}H=H_{2}\oplus\cdots\oplus H_{m},
\]
we have 
\[
V_{t}=\begin{pmatrix}V^{\left(2\right)}_{t}\\
\vdots\\
V^{\left(m\right)}_{t}
\end{pmatrix}.
\]
Thus 
\[
E_{s,t}=\sum^{m}_{k=2}E^{\left(k\right)}_{s,t},\qquad E^{\left(k\right)}_{s,t}=\left(V^{\left(k\right)}_{s}\right)^{*}V^{\left(k\right)}_{t}.
\]

Here $H_{1}=PH$ is the observed block, while $H_{2},\ldots,H_{m}$
are the hidden blocks. The maps $V^{\left(k\right)}_{t}$ are the
components of the leakage map $V_{t}$ landing in the individual hidden
blocks. 

For $k\geq2$, let $d_{k}$ be the graph distance from $1$ to $k$
in $G$, with $d_{k}=\infty$ if there is no path. \prettyref{prop:3-7}
gives the first possible short-time order of each such component in
terms of the block support of $A$. If the $k$-th hidden block is
at graph distance $d_{k}$ from the observed block $H_{1}$, then
$V^{\left(k\right)}_{t}$ has no terms of order below $t^{d_{k}}$.
Consequently, $E^{\left(k\right)}_{s,t}$ has no terms below the scale
$s^{d_{k}}t^{d_{k}}$. 
\begin{prop}
\label{prop:3-7} Assume that $A\in B\left(H\right)$, $A\geq0$,
and that $A$ has block support in $G$. Let $d_{k}$ be the graph
distance from $1$ to $k$. If $d_{k}<\infty$, then, in operator
norm, 
\begin{equation}
V^{\left(k\right)}_{t}=\frac{\left(-t\right)^{d_{k}}}{d_{k}!}P_{k}A^{d_{k}}P_{1}|_{H_{1}}+O\left(t^{d_{k}+1}\right)\label{eq:3-5}
\end{equation}
as $t\downarrow0$. Consequently, as $\left(s,t\right)\to\left(0,0\right)$
with $s,t>0$, 
\begin{equation}
E^{\left(k\right)}_{s,t}=\frac{s^{d_{k}}t^{d_{k}}}{\left(d_{k}!\right)^{2}}\left(P_{k}A^{d_{k}}P_{1}|_{H_{1}}\right)^{*}\left(P_{k}A^{d_{k}}P_{1}|_{H_{1}}\right)+O\left(s^{d_{k}+1}t^{d_{k}}+s^{d_{k}}t^{d_{k}+1}\right).\label{eq:3-6}
\end{equation}
If $d_{k}=\infty$, then 
\begin{equation}
V^{\left(k\right)}_{t}=0\label{eq:3-7}
\end{equation}
for all $t\geq0$. 
\end{prop}

\begin{proof}
Since $A$ is bounded, 
\[
e^{-tA}=\sum^{\infty}_{n=0}\frac{\left(-t\right)^{n}}{n!}A^{n}
\]
in operator norm, uniformly for $t$ in bounded intervals.

We first show that 
\[
P_{k}A^{n}P_{1}=0
\]
whenever $n<d_{k}$. For $n\geq1$, the block expansion of $A^{n}$
gives 
\[
P_{k}A^{n}P_{1}=\sum_{i_{1},\ldots,i_{n-1}}A_{ki_{n-1}}A_{i_{n-1}i_{n-2}}\cdots A_{i_{1}1}.
\]
A summand can be nonzero only if each successive pair among 
\[
1,i_{1},\ldots,i_{n-1},k
\]
is either equal or adjacent in $G$. Deleting consecutive repetitions
gives a path from $1$ to $k$ of length at most $n$. If $n<d_{k}$,
no such path exists. Hence 
\[
P_{k}A^{n}P_{1}=0,\qquad n<d_{k}.
\]

If $d_{k}<\infty$, the exponential series gives 
\[
P_{k}e^{-tA}P_{1}=\frac{\left(-t\right)^{d_{k}}}{d_{k}!}P_{k}A^{d_{k}}P_{1}+\sum^{\infty}_{n=d_{k}+1}\frac{\left(-t\right)^{n}}{n!}P_{k}A^{n}P_{1}.
\]
The remaining sum is $O\left(t^{d_{k}+1}\right)$ in operator norm.
Restricting to $H_{1}$ gives \prettyref{eq:3-5}.

Put 
\[
M_{k}=P_{k}A^{d_{k}}P_{1}|_{H_{1}}.
\]
By \prettyref{eq:3-5}, 
\[
V^{\left(k\right)}_{t}=\frac{\left(-t\right)^{d_{k}}}{d_{k}!}M_{k}+R_{k}\left(t\right),\qquad\left\Vert R_{k}\left(t\right)\right\Vert =O\left(t^{d_{k}+1}\right).
\]
Using 
\[
E^{\left(k\right)}_{s,t}=\left(V^{\left(k\right)}_{s}\right)^{*}V^{\left(k\right)}_{t},
\]
we get 
\[
\begin{aligned}E^{\left(k\right)}_{s,t} & =\frac{s^{d_{k}}t^{d_{k}}}{\left(d_{k}!\right)^{2}}M^{*}_{k}M_{k}+O\left(s^{d_{k}+1}t^{d_{k}}+s^{d_{k}}t^{d_{k}+1}\right).\end{aligned}
\]
This is \prettyref{eq:3-6}.

If $d_{k}=\infty$, the same block expansion gives 
\[
P_{k}A^{n}P_{1}=0
\]
for every $n\geq0$. Hence $P_{k}e^{-tA}P_{1}=0$ for all $t\geq0$,
which proves \prettyref{eq:3-7}. 
\end{proof}

\begin{rem*}
Thus the quadratic coefficient from \prettyref{prop:3-5} comes only
from blocks adjacent to the observed block. A block at distance $d$
can first appear at the scale $s^{d}t^{d}$, with coefficient 
\[
\frac{1}{\left(d!\right)^{2}}\left(P_{k}A^{d}P_{1}|_{H_{1}}\right)^{*}\left(P_{k}A^{d}P_{1}|_{H_{1}}\right).
\]
If $P_{k}A^{d}P_{1}|_{H_{1}}=0$, then the first nonzero contribution
occurs at a higher order. The two-block case treats all hidden directions
as adjacent to $K$, so all hidden leakage is measured at the same
quadratic scale.
\end{rem*}
The distance $d_{k}$ gives the first order allowed by the block support.
The first nonzero term may occur later. Define 
\[
r_{k}=\inf\left\{ n\geq0:P_{k}A^{n}P_{1}|_{H_{1}}\neq0\right\} ,
\]
with the convention that $r_{k}=\infty$ if the set is empty. By the
proof of \prettyref{prop:3-7}, 
\[
r_{k}\geq d_{k}.
\]
If $r_{k}<\infty$, the same argument gives 
\[
V^{\left(k\right)}_{t}=\frac{\left(-t\right)^{r_{k}}}{r_{k}!}P_{k}A^{r_{k}}P_{1}|_{H_{1}}+O\left(t^{r_{k}+1}\right)
\]
in operator norm. Hence 
\[
E^{\left(k\right)}_{s,t}=\frac{s^{r_{k}}t^{r_{k}}}{\left(r_{k}!\right)^{2}}\left(P_{k}A^{r_{k}}P_{1}|_{H_{1}}\right)^{*}\left(P_{k}A^{r_{k}}P_{1}|_{H_{1}}\right)+O\left(s^{r_{k}+1}t^{r_{k}}+s^{r_{k}}t^{r_{k}+1}\right).
\]
Thus $d_{k}$ is a support bound, while $r_{k}$ is the actual order
seen by the $k$-th component.
\begin{cor}
\label{cor:3-8} For $r\geq1$, let 
\[
Q_{r}=\sum_{d_{k}=r}P_{k},
\]
where the sum is over $k\in\left\{ 2,\ldots,m\right\} $ with $d_{k}=r$,
and put 
\[
V^{[r]}_{t}=Q_{r}e^{-tA}P_{1}|_{H_{1}}.
\]
Then 
\[
V_{t}=\sum_{r\geq1}V^{[r]}_{t},
\]
with orthogonal ranges for different $r$, and 
\[
E_{s,t}=\sum_{r\geq1}\left(V^{[r]}_{s}\right)^{*}V^{[r]}_{t}.
\]
For every $r$ with $Q_{r}\neq0$, 
\[
V^{[r]}_{t}=\frac{\left(-t\right)^{r}}{r!}Q_{r}A^{r}P_{1}|_{H_{1}}+O\left(t^{r+1}\right)
\]
in operator norm. Consequently, 
\[
\left(V^{[r]}_{s}\right)^{*}V^{[r]}_{t}=\frac{s^{r}t^{r}}{\left(r!\right)^{2}}\left(Q_{r}A^{r}P_{1}|_{H_{1}}\right)^{*}\left(Q_{r}A^{r}P_{1}|_{H_{1}}\right)+O\left(s^{r+1}t^{r}+s^{r}t^{r+1}\right).
\]
\end{cor}

\begin{proof}
Since the projections $Q_{r}$ have mutually orthogonal ranges and
sum to $P^{\perp}$, we have 
\[
V_{t}=P^{\perp}e^{-tA}P_{1}|_{H_{1}}=\sum_{r\geq1}Q_{r}e^{-tA}P_{1}|_{H_{1}}=\sum_{r\geq1}V^{[r]}_{t}.
\]
The sum is finite. The orthogonality of the ranges gives 
\[
\left(V^{[r]}_{s}\right)^{*}V^{[q]}_{t}=0
\]
whenever $r\neq q$. Hence 
\[
E_{s,t}=V^{*}_{s}V_{t}=\sum_{r\geq1}\left(V^{[r]}_{s}\right)^{*}V^{[r]}_{t}.
\]

For $Q_{r}\neq0$, summing \prettyref{eq:3-5} over the indices $k$
with $d_{k}=r$ gives 
\[
V^{[r]}_{t}=\frac{\left(-t\right)^{r}}{r!}Q_{r}A^{r}P_{1}|_{H_{1}}+O\left(t^{r+1}\right).
\]
Multiplying this expansion at $s$ and $t$ gives the stated expansion
for $\left(V^{[r]}_{s}\right)^{*}V^{[r]}_{t}$. 
\end{proof}

This gives a distance filtration of the compression covariance. The
ordinary quadratic coefficient comes from the distance-one layer.
A block at distance $r$ can first contribute at the scale $s^{r}t^{r}$,
unless the corresponding coefficient vanishes.
\begin{example}
Suppose 
\[
H=H_{1}\oplus H_{2}\oplus H_{3}
\]
and $A\in B\left(H\right)$, $A\geq0$, has block support in the path
\[
1-2-3.
\]
Then $A_{31}=A_{13}=0$. For the observed space $H_{1}$, the $H_{2}$
component has distance one, while the $H_{3}$ component has distance
two. Hence 
\[
V^{\left(2\right)}_{t}=-tP_{2}AP_{1}|_{H_{1}}+O\left(t^{2}\right),
\]
and 
\[
V^{\left(3\right)}_{t}=\frac{t^{2}}{2}P_{3}A^{2}P_{1}|_{H_{1}}+O\left(t^{3}\right).
\]
Thus 
\[
E^{\left(2\right)}_{s,t}=st\left(P_{2}AP_{1}|_{H_{1}}\right)^{*}\left(P_{2}AP_{1}|_{H_{1}}\right)+O\left(s^{2}t+st^{2}\right),
\]
whereas 
\[
E^{\left(3\right)}_{s,t}=\frac{s^{2}t^{2}}{4}\left(P_{3}A^{2}P_{1}|_{H_{1}}\right)^{*}\left(P_{3}A^{2}P_{1}|_{H_{1}}\right)+O\left(s^{3}t^{2}+s^{2}t^{3}\right).
\]
The second hidden block is therefore invisible at the quadratic scale
unless it is directly coupled to $H_{1}$. 
\end{example}

The path case includes the finite Jacobi, or nearest-neighbor chain,
model.
\begin{example}
\label{exa:3-10} Let $H=\mathbb{C}^{m}$, let $H_{j}=\mathbb{C}e_{j}$,
and let $P$ be the projection onto $H_{1}$. Suppose that $A\in B\left(H\right)$
is a positive self-adjoint tridiagonal matrix, 
\[
A=\begin{pmatrix}b_{1} & a_{1} & 0 & \cdots & 0\\
\overline{a_{1}} & b_{2} & a_{2} & \cdots & 0\\
0 & \overline{a_{2}} & b_{3} & \ddots & 0\\
\vdots & \vdots & \ddots & \ddots & a_{m-1}\\
0 & 0 & 0 & \overline{a_{m-1}} & b_{m}
\end{pmatrix}.
\]
Then $A$ has block support in the path graph 
\[
1-2-\cdots-m.
\]
The observed block is the first site $H_{1}$, and the hidden block
$H_{k}$ is at graph distance $k-1$ from it. Thus 
\[
P_{k}e^{-tA}P_{1}|_{H_{1}}=\frac{\left(-t\right)^{k-1}}{\left(k-1\right)!}P_{k}A^{k-1}P_{1}|_{H_{1}}+O\left(t^{k}\right).
\]
In this tridiagonal case there is only one shortest path from $1$
to $k$, and hence 
\[
P_{k}A^{k-1}P_{1}|_{H_{1}}=\overline{a_{k-1}}\cdots\overline{a_{1}}\,\left|e_{k}\left\rangle \right\langle e_{1}\right|.
\]
Therefore 
\[
P_{k}e^{-tA}e_{1}=\frac{\left(-t\right)^{k-1}}{\left(k-1\right)!}\overline{a_{k-1}}\cdots\overline{a_{1}}\,e_{k}+O\left(t^{k}\right).
\]
Consequently, 
\[
E^{\left(k\right)}_{s,t}=\frac{s^{k-1}t^{k-1}}{\left(\left(k-1\right)!\right)^{2}}\left|a_{k-1}\cdots a_{1}\right|^{2}\left|e_{1}\left\rangle \right\langle e_{1}\right|+O\left(s^{k}t^{k-1}+s^{k-1}t^{k}\right).
\]
If $a_{k-1}\cdots a_{1}\neq0$, then the contribution from the $k$-th
site first appears at the scale $s^{k-1}t^{k-1}$. Thus the short-time
covariance detects the nearest-neighbor structure through the order
at which each hidden component first appears.
\end{example}

The preceding examples also indicate what changes in infinite-dimensional
nearest-neighbor models. If the Jacobi coefficients are uniformly
bounded, then the corresponding operator on $\ell^{2}\left(\mathbb{N}\right)$
is bounded, and the same argument applies to each fixed site. Thus,
for a bounded infinite Jacobi matrix, the component of $e^{-tA}e_{1}$
at the $k$-th site still has no term before order $t^{k-1}$.

For genuinely unbounded Jacobi operators, this conclusion is no longer
an operator-norm Taylor statement. The expansion of $e^{-tA}$ at
$t=0$ does not hold in operator norm, and the vectors involved must
be controlled by the domains of powers of $A$. If $e_{1}\in\mathcal{D}\left(A^{N}\right)$,
then one may still read the first $N$ coefficients weakly, or on
the finite span generated by $e_{1},Ae_{1},\ldots,A^{N}e_{1}$. In
that sense the same nearest-neighbor constraint remains visible at
finite order. But it is no longer a uniform bounded-block statement
of the form used above.

The bounded path model also has an intrinsic form. It need not be
chosen in advance from a graph. Assume that $A\in B\left(H\right)$
is self-adjoint. Starting from the observed subspace $PH$, form the
increasing Krylov subspaces
\[
\mathcal{K}_{n}=\overline{span}\left\{ A^{j}PH:0\leq j\leq n\right\} ,\qquad n\geq0,
\]
and define 
\[
\mathcal{H}_{0}=PH,\qquad\mathcal{H}_{n}=\mathcal{K}_{n}\ominus\mathcal{K}_{n-1},\qquad n\geq1.
\]
Let $Q_{n}$ be the orthogonal projection onto $\mathcal{H}_{n}$.
\begin{thm}
\label{thm:3-11} With the notation above, 
\[
A\mathcal{H}_{n}\subseteq\mathcal{H}_{n-1}\oplus\mathcal{H}_{n}\oplus\mathcal{H}_{n+1}
\]
for every $n\geq0$, with the convention that $\mathcal{H}_{-1}=\{0\}$.
In particular, the restriction of $A$ to 
\[
\mathcal{K}_{\infty}=\overline{\bigcup_{n\geq0}\mathcal{K}_{n}}
\]
is block tridiagonal relative to 
\[
\mathcal{K}_{\infty}=\bigoplus_{n\geq0}\mathcal{H}_{n}.
\]
Moreover, for every $n\geq1$, 
\[
Q_{n}e^{-tA}P|_{PH}=\frac{\left(-t\right)^{n}}{n!}Q_{n}A^{n}P|_{PH}+O\left(t^{n+1}\right)
\]
in operator norm as $t\downarrow0$, and 
\[
\overline{ran\left(Q_{n}A^{n}P|_{PH}\right)}=\mathcal{H}_{n}.
\]
\end{thm}

\begin{proof}
The subspaces $\mathcal{K}_{n}$ are increasing, and 
\[
A\mathcal{K}_{n}\subseteq\mathcal{K}_{n+1}.
\]
Let $x\in\mathcal{H}_{n}$. Then $x\in\mathcal{K}_{n}$, so $Ax\in\mathcal{K}_{n+1}$.
If $n\geq2$ and $y\in\mathcal{K}_{n-2}$, then $Ay\in\mathcal{K}_{n-1}$.
Since $x\perp\mathcal{K}_{n-1}$, self-adjointness of $A$ gives 
\[
\left\langle Ax,y\right\rangle =\left\langle x,Ay\right\rangle =0.
\]
Therefore $Ax\perp\mathcal{K}_{n-2}$ when $n\geq2$. It follows that
\[
Ax\in\mathcal{K}_{n+1}\ominus\mathcal{K}_{n-2}=\mathcal{H}_{n-1}\oplus\mathcal{H}_{n}\oplus\mathcal{H}_{n+1}.
\]
For $n=0$ and $n=1$, the same conclusion follows directly from $A\mathcal{K}_{n}\subseteq\mathcal{K}_{n+1}$.
Hence 
\[
A\mathcal{H}_{n}\subseteq\mathcal{H}_{n-1}\oplus\mathcal{H}_{n}\oplus\mathcal{H}_{n+1}
\]
for every $n\geq0$.

The subspace $\mathcal{K}_{\infty}$ is invariant under $A$. Since
$A$ is self-adjoint, $\mathcal{K}_{\infty}$ reduces $A$. The block
tridiagonal form follows from the inclusion just proved.

For the short-time expansion, use 
\[
e^{-tA}=\sum^{n}_{j=0}\frac{\left(-t\right)^{j}}{j!}A^{j}+O\left(t^{n+1}\right)
\]
in operator norm. If $j<n$, then $A^{j}PH\subseteq\mathcal{K}_{n-1}$,
and hence 
\[
Q_{n}A^{j}P|_{PH}=0.
\]
Therefore 
\[
Q_{n}e^{-tA}P|_{PH}=\frac{\left(-t\right)^{n}}{n!}Q_{n}A^{n}P|_{PH}+O\left(t^{n+1}\right).
\]

Finally, 
\[
\mathcal{K}_{n}=\overline{\mathcal{K}_{n-1}+A^{n}PH}.
\]
Applying $Q_{n}$ gives 
\[
\mathcal{H}_{n}=\overline{Q_{n}A^{n}PH}=\overline{ran\left(Q_{n}A^{n}P|_{PH}\right)}.
\]
\end{proof}

\begin{rem}
\label{rem:3-12}The finite Jacobi model in \prettyref{exa:3-10}
is the rank-one case of this construction. If $a_{1}\cdots a_{m-1}\neq0$,
then 
\[
\mathcal{K}_{n}=span\left\{ e_{1},\ldots,e_{n+1}\right\} ,\qquad\mathcal{H}_{n}=\mathbb{C}e_{n+1},
\]
for $0\leq n\leq m-1$. Thus $Q_{n}=P_{n+1}$, and the expansion in
\prettyref{thm:3-11} becomes the formula in \prettyref{exa:3-10}
with $k=n+1$. If one of the coefficients $a_{j}$ vanishes, the Krylov
layers stop or collapse accordingly; the construction follows the
part of the tridiagonal model actually generated from the observed
vector.
\end{rem}

The bounded constructions in this section rely on Taylor expansion
in operator norm, and therefore produce integer powers such as $s^{r}t^{r}$.
In \prettyref{sec:4}, we pass instead to rescaled covariance limits.
This keeps the semigroup identities but does not require a bounded
generator block, and the leading scale need not come from a Taylor
coefficient.

\section{Leakage tangents}\label{sec:4}

The preceding higher-order results used boundedness of $A$ in an
essential way. The operator-norm Taylor expansion of $e^{-tA}$ gave
integer-order coefficients for the leakage maps. We now return to
the general setting where $A\geq0$ is self-adjoint. The generator
may be unbounded, and an operator-norm Taylor expansion at $t=0$
need not exist.

Instead of extracting coefficients of $e^{-tA}$, we study rescaled
short-time limits of the compression covariance itself. The semigroup
identities from \prettyref{sec:3} survive under this rescaling. In
particular, the lower-right dynamics still acts on the limiting kernel
space. A tangent of $E_{s,t}$ therefore comes with more than positivity.
It carries a canonical semigroup on its Kolmogorov space.

\subsection{Tangent kernels and induced semigroups}

Given $\left(A,P\right)$, where $A\geq0$ is self-adjoint on a Hilbert
space $H$ and $P$ is an orthogonal projection, let 
\[
T_{t}:=e^{-tA},\qquad t\geq0,
\]
be the associated self-adjoint contraction semigroup, and set 
\[
K:=PH,\qquad V_{t}:=P^{\perp}T_{t}P|_{K},\qquad E_{s,t}:=V^{*}_{s}V_{t}.
\]

Let $a:\left(0,1\right)\to\left(0,\infty\right)$ be a scale function
with $a\left(\varepsilon\right)\to0$ as $\varepsilon\downarrow0$.
For $\varepsilon>0$, put 
\[
W_{\varepsilon}\left(t\right)=a\left(\varepsilon\right)^{-1/2}V_{\varepsilon t},\qquad t\geq0.
\]
Thus 
\[
W_{\varepsilon}\left(s\right)^{*}W_{\varepsilon}\left(t\right)=a\left(\varepsilon\right)^{-1}E_{\varepsilon s,\varepsilon t}.
\]

\begin{defn}
\label{def:4-1} A $B\left(K\right)$-valued kernel $F$ on $\left[0,\infty\right)\times\left[0,\infty\right)$
is called an $a$-scale leakage tangent for $\left(A,P\right)$ if
\begin{equation}
a\left(\varepsilon\right)^{-1}E_{\varepsilon s,\varepsilon t}\to F\left(s,t\right)\label{eq:4-1}
\end{equation}
locally uniformly in $B\left(K\right)$ as $\varepsilon\downarrow0$. 
\end{defn}

Every leakage tangent is positive, since each kernel 
\[
\left(s,t\right)\mapsto a\left(\varepsilon\right)^{-1}E_{\varepsilon s,\varepsilon t}
\]
is positive. Also, 
\[
F\left(0,t\right)=F\left(t,0\right)=0
\]
for all $t\geq0$, because $V_{0}=0$.

Let $H_{F}$ be the Kolmogorov space of the positive kernel $F$.
That is, $H_{F}$ is the completion of the algebraic span of symbols
$\left[t,h\right]$, $t\geq0$, $h\in K$, with inner product 
\[
\left\langle \left[s,h\right],\left[t,k\right]\right\rangle =\left\langle h,F\left(s,t\right)k\right\rangle .
\]
We write 
\[
w_{t}:K\to H_{F},\qquad w_{t}h=\left[t,h\right].
\]
Then 
\[
F\left(s,t\right)=w^{*}_{s}w_{t}.
\]

\prettyref{thm:4-2} is a main structural characterization. It shows
that the hidden block from \prettyref{sec:3} induces a semigroup
on the tangent Kolmogorov space $H_{F}$.
\begin{thm}
\label{thm:4-2} Assume that $F$ is an $a$-scale leakage tangent
for $\left(A,P\right)$. For each $r\geq0$, define on the algebraic
span of $\left\{ w_{t}h:t\geq0,\ h\in K\right\} $ by 
\begin{equation}
S_{r}w_{t}h=w_{r+t}h-w_{r}h.\label{eq:4-2}
\end{equation}
Then $S_{r}$ is well-defined and extends to a positive self-adjoint
contraction on $H_{F}$. Moreover, 
\[
S_{0}=I_{H_{F}},\qquad S_{r}S_{q}=S_{r+q}
\]
for all $r,q\geq0$. If $F$ is locally norm-continuous, then $\left(S_{r}\right)_{r\geq0}$
is strongly continuous.

Consequently, 
\begin{equation}
w_{r+t}=w_{r}+S_{r}w_{t}\label{eq:4-3}
\end{equation}
for all $r,t\geq0$. 
\end{thm}

\begin{proof}
For $\varepsilon>0$, recall that 
\[
W_{\varepsilon}\left(t\right)=a\left(\varepsilon\right)^{-1/2}V_{\varepsilon t}.
\]
By \prettyref{lem:3-2}, 
\[
D_{\varepsilon r}V_{\varepsilon t}=V_{\varepsilon\left(r+t\right)}-V_{\varepsilon r}C_{\varepsilon t}.
\]
Multiplying by $a\left(\varepsilon\right)^{-1/2}$ gives 
\begin{equation}
D_{\varepsilon r}W_{\varepsilon}\left(t\right)=W_{\varepsilon}\left(r+t\right)-W_{\varepsilon}\left(r\right)C_{\varepsilon t}.\label{eq:4-4}
\end{equation}

Let 
\[
x=\sum^{n}_{i=1}w_{t_{i}}h_{i}
\]
be a finite linear combination. Put 
\[
x_{\varepsilon}=\sum^{n}_{i=1}W_{\varepsilon}\left(t_{i}\right)h_{i}
\]
and 
\[
y_{\varepsilon}=\sum^{n}_{i=1}\left(W_{\varepsilon}\left(r+t_{i}\right)h_{i}-W_{\varepsilon}\left(r\right)h_{i}\right).
\]
The local uniform convergence in \prettyref{eq:4-1} gives 
\[
\left\Vert x_{\varepsilon}\right\Vert ^{2}\to\left\Vert x\right\Vert ^{2}
\]
and 
\[
\left\Vert y_{\varepsilon}\right\Vert ^{2}\to\left\Vert \sum^{n}_{i=1}\left(w_{r+t_{i}}h_{i}-w_{r}h_{i}\right)\right\Vert ^{2}.
\]

By \prettyref{eq:4-4}, 
\[
D_{\varepsilon r}x_{\varepsilon}=\sum^{n}_{i=1}\left(W_{\varepsilon}\left(r+t_{i}\right)h_{i}-W_{\varepsilon}\left(r\right)C_{\varepsilon t_{i}}h_{i}\right).
\]
Thus 
\[
y_{\varepsilon}-D_{\varepsilon r}x_{\varepsilon}=W_{\varepsilon}\left(r\right)\sum^{n}_{i=1}\left(C_{\varepsilon t_{i}}-I_{K}\right)h_{i}.
\]
The operators $W_{\varepsilon}\left(r\right)$ are bounded uniformly
for small $\varepsilon$, because 
\[
\left\Vert W_{\varepsilon}\left(r\right)\right\Vert ^{2}=\left\Vert W_{\varepsilon}\left(r\right)^{*}W_{\varepsilon}\left(r\right)\right\Vert =\left\Vert a\left(\varepsilon\right)^{-1}E_{\varepsilon r,\varepsilon r}\right\Vert 
\]
has a finite limit. Also $C_{\varepsilon t_{i}}h_{i}\to h_{i}$ for
each $i$, since $\left(C_{t}\right)_{t\geq0}$ is strongly continuous
and $C_{0}=I_{K}$. Hence 
\[
\left\Vert y_{\varepsilon}-D_{\varepsilon r}x_{\varepsilon}\right\Vert \to0.
\]

Since $D_{\varepsilon r}$ is a contraction on $P^{\perp}H$, 
\[
\left\Vert y_{\varepsilon}\right\Vert =\left\Vert D_{\varepsilon r}x_{\varepsilon}+o\left(1\right)\right\Vert \leq\left\Vert x_{\varepsilon}\right\Vert +o\left(1\right).
\]
Passing to the limit gives 
\[
\left\Vert \sum^{n}_{i=1}\left(w_{r+t_{i}}h_{i}-w_{r}h_{i}\right)\right\Vert \leq\left\Vert \sum^{n}_{i=1}w_{t_{i}}h_{i}\right\Vert .
\]
Therefore \prettyref{eq:4-2} is well-defined on the quotient and
extends to a contraction on $H_{F}$.

We next prove positivity. Since $D_{\varepsilon r}\geq0$, 
\[
\left\langle x_{\varepsilon},D_{\varepsilon r}x_{\varepsilon}\right\rangle \geq0.
\]
Using again 
\[
y_{\varepsilon}-D_{\varepsilon r}x_{\varepsilon}\to0
\]
and the convergence of the Gram kernels, we obtain 
\[
\left\langle x,S_{r}x\right\rangle =\lim_{\varepsilon\downarrow0}\left\langle x_{\varepsilon},D_{\varepsilon r}x_{\varepsilon}\right\rangle \geq0.
\]
Thus $S_{r}\geq0$.

The same approximation gives self-adjointness. For finite sums 
\[
x=\sum^{n}_{i=1}w_{t_{i}}h_{i},\qquad z=\sum^{m}_{j=1}w_{u_{j}}k_{j},
\]
let $x_{\varepsilon}$ and $z_{\varepsilon}$ be the corresponding
lifts. Since $D_{\varepsilon r}=D^{*}_{\varepsilon r}$, 
\[
\left\langle S_{r}x,z\right\rangle =\lim_{\varepsilon\downarrow0}\left\langle D_{\varepsilon r}x_{\varepsilon},z_{\varepsilon}\right\rangle =\lim_{\varepsilon\downarrow0}\left\langle x_{\varepsilon},D_{\varepsilon r}z_{\varepsilon}\right\rangle =\left\langle x,S_{r}z\right\rangle .
\]
Hence $S_{r}$ is self-adjoint.

It remains to prove the semigroup law. On generators, 
\[
\begin{aligned}S_{r}S_{q}w_{t}h & =S_{r}\left(w_{q+t}h-w_{q}h\right)\\
 & =\left(w_{r+q+t}h-w_{r}h\right)-\left(w_{r+q}h-w_{r}h\right)\\
 & =w_{r+q+t}h-w_{r+q}h\\
 & =S_{r+q}w_{t}h.
\end{aligned}
\]
Thus $S_{r}S_{q}=S_{r+q}$ on a dense subspace, and hence on $H_{F}$.
Also $S_{0}=I_{H_{F}}$, because $w_{0}=0$.

Finally, assume that $F$ is locally norm-continuous. For a generator
$w_{t}h$, 
\[
S_{r}w_{t}h-S_{q}w_{t}h=\left(w_{r+t}h-w_{r}h\right)-\left(w_{q+t}h-w_{q}h\right).
\]
The norm of this vector is a finite expression in the values of $F$
at the points 
\[
r+t,\ r,\ q+t,\ q.
\]
Hence it tends to zero as $r\to q$. Since the operators $S_{r}$
are contractions, strong continuity follows on all of $H_{F}$. The
identity \prettyref{eq:4-3} is just \prettyref{eq:4-2} rewritten. 
\end{proof}

\begin{rem}
\label{rem:4-3}The semigroup $\left(S_{r}\right)_{r\geq0}$ is part
of the tangent object, not imposed externally. It is forced by the
lower-right block dynamics 
\[
D_{r}V_{t}=V_{r+t}-V_{r}C_{t}.
\]
Thus a leakage tangent is the covariance kernel of an additive cocycle
for a self-adjoint contraction semigroup. 
\end{rem}

\prettyref{thm:4-2} gives intrinsic restrictions on which positive
kernels can arise as leakage tangents. For $r\geq0$, define the shifted
increment kernel 
\begin{equation}
F^{[r]}\left(s,t\right)=F\left(r+s,r+t\right)-F\left(r+s,r\right)-F\left(r,r+t\right)+F\left(r,r\right).\label{eq:4-5}
\end{equation}

\begin{cor}
\label{cor:4-4} For every $r\geq0$, 
\begin{equation}
0\leq F^{[r]}\leq F\label{eq:4-6}
\end{equation}
in the order of positive $B\left(K\right)$-valued kernels. Moreover,
for all $r,s,t\geq0$, 
\begin{equation}
F\left(r+s,t\right)-F\left(r,t\right)=F\left(s,r+t\right)-F\left(s,r\right).\label{eq:4-7}
\end{equation}
\end{cor}

\begin{proof}
By \prettyref{eq:4-3}, 
\[
w_{r+s}-w_{r}=S_{r}w_{s}.
\]
Therefore 
\[
F^{[r]}\left(s,t\right)=w^{*}_{s}S^{2}_{r}w_{t}.
\]
Since $0\leq S_{r}\leq I_{H_{F}}$, the functional calculus gives
\[
0\leq S^{2}_{r}\leq I_{H_{F}}.
\]
Hence \prettyref{eq:4-6} holds.

For \prettyref{eq:4-7}, use the self-adjointness of $S_{r}$: 
\[
\left\langle S_{r}w_{s}h,w_{t}k\right\rangle =\left\langle w_{s}h,S_{r}w_{t}k\right\rangle .
\]
Using $S_{r}w_{s}=w_{r+s}-w_{r}$ and $S_{r}w_{t}=w_{r+t}-w_{r}$
gives 
\[
\left\langle h,\left(F\left(r+s,t\right)-F\left(r,t\right)\right)k\right\rangle =\left\langle h,\left(F\left(s,r+t\right)-F\left(s,r\right)\right)k\right\rangle .
\]
Since $h,k$ are arbitrary, \prettyref{eq:4-7} follows. 
\end{proof}

\begin{defn}
\label{def:4-5} Let $K$ be a Hilbert space. A positive $B\left(K\right)$-valued
kernel $F$ on $\left[0,\infty\right)\times\left[0,\infty\right)$
with $F\left(0,t\right)=F\left(t,0\right)=0$ is called a \emph{cocycle
kernel} if, on its Kolmogorov space $H_{F}$, the formula 
\[
S_{r}w_{t}h=w_{r+t}h-w_{r}h
\]
defines a positive self-adjoint contraction semigroup $\left(S_{r}\right)_{r\geq0}$. 
\end{defn}

\begin{cor}
\label{cor:4-6}When the induced semigroup $\left(S_{r}\right)_{r\geq0}$
is strongly continuous, the spectral theorem gives a positive self-adjoint
operator $B$ on $H_{F}$ such that 
\begin{equation}
S_{r}=e^{-rB},\qquad r\geq0.\label{eq:4-8}
\end{equation}
In this notation, the tangent cocycle identity becomes 
\begin{equation}
w_{r+t}=w_{r}+e^{-rB}w_{t}.\label{eq:4-9}
\end{equation}
The pair $\left(B,w\right)$ is the hidden tangent data associated
with $F$.
\end{cor}

The preceding conclusions are operator-valued and do not require differentiability
of the tangent kernel. In the scalar smooth case, they reduce to a
more familiar condition: the mixed derivative of the tangent kernel
is a positive Hankel kernel.
\begin{prop}
\label{prop:4-7} Suppose $K=\mathbb{C}$ and $F$ is twice continuously
differentiable on $\left(0,\infty\right)^{2}$. Then there is a function
$\varphi$ on $\left(0,\infty\right)$ such that 
\[
\frac{\partial^{2}F}{\partial s\partial t}\left(s,t\right)=\varphi\left(s+t\right).
\]
Moreover, the kernel 
\[
\left(s,t\right)\mapsto\varphi\left(s+t\right)
\]
is positive.
\end{prop}

\begin{proof}
Differentiating the transport identity \prettyref{eq:4-7} in $s$
and $t$ gives 
\[
\frac{\partial^{2}F}{\partial s\partial t}\left(r+s,t\right)=\frac{\partial^{2}F}{\partial s\partial t}\left(s,r+t\right).
\]
Hence $\frac{\partial^{2}F}{\partial s\partial t}\left(s,t\right)$
depends only on $s+t$. Write 
\[
\frac{\partial^{2}F}{\partial s\partial t}\left(s,t\right)=\varphi\left(s+t\right).
\]

For $\delta>0$, define 
\[
G_{\delta}\left(s,t\right)=F\left(s+\delta,t+\delta\right)-F\left(s+\delta,t\right)-F\left(s,t+\delta\right)+F\left(s,t\right).
\]
Using the cocycle identity, 
\[
w_{s+\delta}-w_{s}=S_{s}w_{\delta}.
\]
Therefore 
\[
G_{\delta}\left(s,t\right)=\left(w_{s+\delta}-w_{s}\right)^{*}\left(w_{t+\delta}-w_{t}\right)=w^{*}_{\delta}S_{s+t}w_{\delta}.
\]
The kernel $G_{\delta}$ is positive. Indeed, for $s_{1},\ldots,s_{n}>0$
and $c_{1},\ldots,c_{n}\in\mathbb{C}$, 
\[
\sum^{n}_{i,j=1}\overline{c_{i}}c_{j}G_{\delta}\left(s_{i},s_{j}\right)=\Big\Vert\sum^{n}_{i=1}c_{i}S_{s_{i}}w_{\delta}\Big\Vert^{2}\geq0.
\]
Since $F$ is $C^{2}$, 
\[
\frac{1}{\delta^{2}}G_{\delta}\left(s,t\right)\to\varphi\left(s+t\right)
\]
locally uniformly on $\left(0,\infty\right)^{2}$ as $\delta\downarrow0$.
Hence $\left(s,t\right)\mapsto\varphi\left(s+t\right)$ is positive. 
\end{proof}

\begin{cor}
\label{cor:4-8} Assume, in addition to the hypotheses of \prettyref{prop:4-7},
that $\varphi$ is completely monotone. Then there is a positive measure
$\nu$ on $\left[0,\infty\right)$ such that 
\[
\varphi\left(u\right)=\int^{\infty}_{0}e^{-ux}\,d\nu\left(x\right),\qquad u>0.
\]
If the integral below is finite for the relevant values of $s$ and
$t$, then 
\[
F\left(s,t\right)=\int^{\infty}_{0}\frac{\left(1-e^{-sx}\right)\left(1-e^{-tx}\right)}{x^{2}}\,d\nu\left(x\right),
\]
with the integrand interpreted as $st$ at $x=0$. 
\end{cor}

\begin{proof}
By Bernstein's theorem for completely monotone functions, $\varphi$
is the Laplace transform of a positive measure $\nu$ on $\left[0,\infty\right)$;
see, for example, \cite{MR2978140} or \cite[Chapter IV]{MR5923}.
Thus 
\[
\frac{\partial^{2}F}{\partial s\partial t}\left(s,t\right)=\int^{\infty}_{0}e^{-\left(s+t\right)x}\,d\nu\left(x\right).
\]
Integrating first in $s$ and then in $t$, and using $F\left(0,t\right)=F\left(s,0\right)=0$,
gives the stated formula. 
\end{proof}

\subsection{Spectral cocycle models}

We now give a class of cocycle kernels coming from positive spectral
data. The construction is operator-valued from the start, and the
fractional kernels appear as a special case. This is a tangent-level
construction: it describes kernels after the rescaled limit has already
been taken. It does not, by itself, solve the separate realization
problem of deciding which cocycle kernels arise from pre-limit compression
data $\left(A,P,a\right)$. That problem would require constructing
a larger semigroup $T_{t}=e^{-tA}$ whose off-diagonal block has the
prescribed cocycle as its tangent.

In the scalar smooth case, \prettyref{cor:4-8} represents the mixed
derivative by a Laplace measure. Here we work directly with the kernel.
The factor produced by integrating twice is absorbed into the positive
spectral measure $\Sigma$.

Let $\Sigma$ be a positive sesquilinear measure on $\left(0,\infty\right)$
with values on $K$, meaning that each $\Sigma_{h,k}$ is a complex
measure, $\Sigma_{h,h}$ is positive, and 
\[
\Sigma_{h,k}\left(\Omega\right)=\overline{\Sigma_{k,h}\left(\Omega\right)}
\]
for all Borel sets $\Omega\subset\left(0,\infty\right)$. Assume that,
for every $s,t\geq0$, the formula 
\begin{equation}
\left\langle h,F_{\Sigma}\left(s,t\right)k\right\rangle =\int^{\infty}_{0}\left(1-e^{-sx}\right)\left(1-e^{-tx}\right)d\Sigma_{h,k}\left(x\right)\label{eq:4-10}
\end{equation}
defines a bounded operator $F_{\Sigma}\left(s,t\right)\in B\left(K\right)$.
\begin{prop}
\label{prop:4-5} The kernel $F_{\Sigma}$ is a cocycle kernel (see
\prettyref{def:4-5}). 
\end{prop}

\begin{proof}
Let $M_{\Sigma}$ be the Hilbert space obtained as follows. Start
with the vector space of simple $K$-valued functions $f$ on $\left(0,\infty\right)$
such that 
\[
\int^{\infty}_{0}d\Sigma_{f\left(x\right),f\left(x\right)}\left(x\right)<\infty.
\]
On this space define 
\[
\left\langle f,g\right\rangle _{M_{\Sigma}}=\int^{\infty}_{0}d\Sigma_{f\left(x\right),g\left(x\right)}\left(x\right).
\]
Quotient by the null space and complete.

For $r\geq0$, define 
\[
S_{r}f\left(x\right)=e^{-rx}f\left(x\right).
\]
Then $\left(S_{r}\right)_{r\geq0}$ is a positive self-adjoint contraction
semigroup on $M_{\Sigma}$.

For $t\geq0$ and $h\in K$, define 
\[
w_{t}h\left(x\right)=\left(1-e^{-tx}\right)h.
\]
Then 
\[
\left\Vert w_{t}h\right\Vert ^{2}_{M_{\Sigma}}=\int^{\infty}_{0}\left(1-e^{-tx}\right)^{2}d\Sigma_{h,h}\left(x\right)=\left\langle h,F_{\Sigma}\left(t,t\right)h\right\rangle .
\]
Hence $w_{t}$ is a bounded map from $K$ into $M_{\Sigma}$. By polarization,
\[
w^{*}_{s}w_{t}=F_{\Sigma}\left(s,t\right).
\]
Moreover, 
\[
1-e^{-\left(r+t\right)x}=\left(1-e^{-rx}\right)+e^{-rx}\left(1-e^{-tx}\right).
\]
Therefore 
\[
w_{r+t}=w_{r}+S_{r}w_{t}.
\]
Hence $F_{\Sigma}$ is a cocycle kernel. 
\end{proof}

The preceding measure model includes the usual spectral theorem picture
as a special case. Let $M$ be a Hilbert space, let $B\geq0$ be self-adjoint
on $M$, and let $R:K\to M$ be bounded. Define 
\[
w_{t}=\left(I_{M}-e^{-tB}\right)R.
\]
Then 
\[
w_{r+t}=w_{r}+e^{-rB}w_{t}
\]
and 
\[
F_{B,R}\left(s,t\right)=R^{*}\left(I_{M}-e^{-sB}\right)\left(I_{M}-e^{-tB}\right)R
\]
is a cocycle kernel. In the spectral representation of $B$, this
is the case where 
\[
d\Sigma_{h,k}\left(x\right)=d\left\langle Rh,E_{B}\left(x\right)Rk\right\rangle .
\]

The measure formulation in \prettyref{prop:4-5} is more flexible
because it allows singular entrance data. In such cases $w_{t}$ is
defined for $t>0$, although there need not be a bounded vector $Rh$
at time zero.
\begin{example}
Let $0<\beta<1$, let $Q\in B\left(K\right)$ with $Q\geq0$, and
set 
\[
d\Sigma_{h,k}\left(x\right)=c_{\beta}\left\langle Q^{1/2}h,Q^{1/2}k\right\rangle x^{-1-\beta}dx,\qquad c_{\beta}=\frac{\beta}{\Gamma\left(1-\beta\right)}.
\]
Then \prettyref{eq:4-10} gives 
\[
F_{\beta,Q}\left(s,t\right)=\left(s^{\beta}+t^{\beta}-\left(s+t\right)^{\beta}\right)Q.
\]
Indeed, 
\[
s^{\beta}+t^{\beta}-\left(s+t\right)^{\beta}=c_{\beta}\int^{\infty}_{0}\left(1-e^{-sx}\right)\left(1-e^{-tx}\right)x^{-1-\beta}dx.
\]
The corresponding Kolmogorov space is 
\[
L^{2}\left(\left(0,\infty\right),x^{-1-\beta}dx\right)\otimes\overline{ranQ^{1/2}},
\]
with 
\[
S_{r}f\left(x\right)=e^{-rx}f\left(x\right)
\]
and 
\[
w_{t}h\left(x\right)=c^{1/2}_{\beta}\left(1-e^{-tx}\right)Q^{1/2}h.
\]
Thus the fractional kernel is a spectral cocycle kernel. It is the
case where the spectral measure has density $c_{\beta}x^{-1-\beta}Q$
in the quadratic form sense. 
\end{example}

\section{Dual tangent covariance}\label{sec:5}

The preceding construction used the upper-left covariance 
\[
E_{s,t}=V^{*}_{s}V_{t}
\]
to obtain the tangent kernel $F$. The lower-right covariance from
\prettyref{sec:3} also has a tangent form, but not as an operator
limit on the original space $P^{\perp}H$. After testing against vectors
whose norms are measured at the tangent scale, the canonical limit
lives on the tangent Kolmogorov space $H_{F}$.

For $s,t\geq0$, define 
\[
\Theta_{s,t}=w_{s}w^{*}_{t}\in B\left(H_{F}\right).
\]
Thus $\Theta$ is the hidden-side covariance kernel associated with
the tangent maps $w_{t}:K\to H_{F}$.
\begin{thm}
\label{thm:5-1} Assume that $F$ is an $a$-scale leakage tangent
for $\left(A,P\right)$. Let $H_{F}$ be its Kolmogorov space, let
$w_{t}:K\to H_{F}$ be the canonical maps, and let $\left(S_{r}\right)_{r\geq0}$
be the semigroup from \prettyref{thm:4-2}. Let 
\[
x=\sum_{i}w_{p_{i}}h_{i},\qquad y=\sum_{j}w_{q_{j}}k_{j}
\]
be finite linear combinations in $H_{F}$, and choose the corresponding
lifts 
\[
x_{\varepsilon}=\sum_{i}W_{\varepsilon}\left(p_{i}\right)h_{i},\qquad y_{\varepsilon}=\sum_{j}W_{\varepsilon}\left(q_{j}\right)k_{j}
\]
in $P^{\perp}H$. Then, for all $s,t\geq0$, 
\begin{equation}
\lim_{\varepsilon\downarrow0}\left\langle x_{\varepsilon},a\left(\varepsilon\right)^{-1}\left(D_{\varepsilon\left(s+t\right)}-D_{\varepsilon s}D_{\varepsilon t}\right)y_{\varepsilon}\right\rangle =\left\langle x,\Theta_{s,t}y\right\rangle .\label{eq:5-1}
\end{equation}
Moreover, for all $r,q,s,t\geq0$, 
\begin{equation}
\lim_{\varepsilon\downarrow0}\left\langle D_{\varepsilon r}x_{\varepsilon},a\left(\varepsilon\right)^{-1}\left(D_{\varepsilon\left(s+t\right)}-D_{\varepsilon s}D_{\varepsilon t}\right)D_{\varepsilon q}y_{\varepsilon}\right\rangle =\left\langle S_{r}x,\Theta_{s,t}S_{q}y\right\rangle .\label{eq:5-2}
\end{equation}
\end{thm}

\begin{proof}
By \prettyref{eq:3-3}, 
\[
D_{\varepsilon\left(s+t\right)}-D_{\varepsilon s}D_{\varepsilon t}=V_{\varepsilon s}V^{*}_{\varepsilon t}.
\]
Since 
\[
W_{\varepsilon}\left(u\right)=a\left(\varepsilon\right)^{-1/2}V_{\varepsilon u},
\]
we have 
\[
a\left(\varepsilon\right)^{-1}\left(D_{\varepsilon\left(s+t\right)}-D_{\varepsilon s}D_{\varepsilon t}\right)=W_{\varepsilon}\left(s\right)W_{\varepsilon}\left(t\right)^{*}.
\]
Therefore 
\[
\begin{aligned} & \left\langle x_{\varepsilon},a\left(\varepsilon\right)^{-1}\left(D_{\varepsilon\left(s+t\right)}-D_{\varepsilon s}D_{\varepsilon t}\right)y_{\varepsilon}\right\rangle \\
 & \qquad=\left\langle x_{\varepsilon},W_{\varepsilon}\left(s\right)W_{\varepsilon}\left(t\right)^{*}y_{\varepsilon}\right\rangle \\
 & \qquad=\left\langle W_{\varepsilon}\left(s\right)^{*}x_{\varepsilon},W_{\varepsilon}\left(t\right)^{*}y_{\varepsilon}\right\rangle .
\end{aligned}
\]
The local uniform convergence of the Gram kernels gives 
\[
W_{\varepsilon}\left(s\right)^{*}x_{\varepsilon}=\sum_{i}W_{\varepsilon}\left(s\right)^{*}W_{\varepsilon}\left(p_{i}\right)h_{i}\to\sum_{i}F\left(s,p_{i}\right)h_{i}=w^{*}_{s}x
\]
in $K$. Similarly, 
\[
W_{\varepsilon}\left(t\right)^{*}y_{\varepsilon}\to w^{*}_{t}y.
\]
Hence the last inner product converges to 
\[
\left\langle w^{*}_{s}x,w^{*}_{t}y\right\rangle =\left\langle x,w_{s}w^{*}_{t}y\right\rangle =\left\langle x,\Theta_{s,t}y\right\rangle .
\]
This proves \prettyref{eq:5-1}.

For \prettyref{eq:5-2}, it is enough to show that 
\[
W_{\varepsilon}\left(s\right)^{*}D_{\varepsilon r}x_{\varepsilon}\to w^{*}_{s}S_{r}x
\]
for every finite lift $x_{\varepsilon}$, and similarly for $y$.
Using \prettyref{eq:4-4}, 
\[
D_{\varepsilon r}W_{\varepsilon}\left(p_{i}\right)=W_{\varepsilon}\left(r+p_{i}\right)-W_{\varepsilon}\left(r\right)C_{\varepsilon p_{i}}.
\]
Thus 
\[
\begin{aligned}W_{\varepsilon}\left(s\right)^{*}D_{\varepsilon r}x_{\varepsilon} & =\sum_{i}W_{\varepsilon}\left(s\right)^{*}W_{\varepsilon}\left(r+p_{i}\right)h_{i}\\
 & \quad-\sum_{i}W_{\varepsilon}\left(s\right)^{*}W_{\varepsilon}\left(r\right)C_{\varepsilon p_{i}}h_{i}.
\end{aligned}
\]
The first sum converges to 
\[
\sum_{i}F\left(s,r+p_{i}\right)h_{i}.
\]
The second sum converges to 
\[
\sum_{i}F\left(s,r\right)h_{i},
\]
because $C_{\varepsilon p_{i}}h_{i}\to h_{i}$ and $W_{\varepsilon}\left(s\right)^{*}W_{\varepsilon}\left(r\right)$
is uniformly bounded for small $\varepsilon$. Therefore 
\[
W_{\varepsilon}\left(s\right)^{*}D_{\varepsilon r}x_{\varepsilon}\to\sum_{i}\left(F\left(s,r+p_{i}\right)-F\left(s,r\right)\right)h_{i}=w^{*}_{s}S_{r}x.
\]
The same argument gives 
\[
W_{\varepsilon}\left(t\right)^{*}D_{\varepsilon q}y_{\varepsilon}\to w^{*}_{t}S_{q}y.
\]
Using again 
\[
a\left(\varepsilon\right)^{-1}\left(D_{\varepsilon\left(s+t\right)}-D_{\varepsilon s}D_{\varepsilon t}\right)=W_{\varepsilon}\left(s\right)W_{\varepsilon}\left(t\right)^{*},
\]
we obtain 
\[
\begin{aligned} & \left\langle D_{\varepsilon r}x_{\varepsilon},a\left(\varepsilon\right)^{-1}\left(D_{\varepsilon\left(s+t\right)}-D_{\varepsilon s}D_{\varepsilon t}\right)D_{\varepsilon q}y_{\varepsilon}\right\rangle \\
 & \qquad=\left\langle W_{\varepsilon}\left(s\right)^{*}D_{\varepsilon r}x_{\varepsilon},W_{\varepsilon}\left(t\right)^{*}D_{\varepsilon q}y_{\varepsilon}\right\rangle \\
 & \qquad\to\left\langle w^{*}_{s}S_{r}x,w^{*}_{t}S_{q}y\right\rangle \\
 & \qquad=\left\langle S_{r}x,w_{s}w^{*}_{t}S_{q}y\right\rangle =\left\langle S_{r}x,\Theta_{s,t}S_{q}y\right\rangle .
\end{aligned}
\]
This proves \prettyref{eq:5-2}. 
\end{proof}

The theorem gives the lower-right covariance in the tangent space.
We now state the corresponding increment identities. They show that
hidden increments are transported by $S_{r}$, and that the visible
increment kernels form a decreasing family.
\begin{cor}
\label{cor:5-2} For $r\geq0$, define 
\[
\Theta^{[r]}_{s,t}=\left(w_{r+s}-w_{r}\right)\left(w_{r+t}-w_{r}\right)^{*}
\]
and 
\[
F^{[r]}\left(s,t\right)=\left(w_{r+s}-w_{r}\right)^{*}\left(w_{r+t}-w_{r}\right).
\]
Then, for all $r,s,t\geq0$, 
\[
\Theta^{[r]}_{s,t}=S_{r}\Theta_{s,t}S_{r}
\]
and 
\[
F^{[r]}\left(s,t\right)=w^{*}_{s}S^{2}_{r}w_{t}.
\]
Consequently, the kernels $F^{[r]}$ form a decreasing family: 
\[
0\leq F^{[r+q]}\leq F^{[r]}\leq F
\]
for all $r,q\geq0$. Moreover, 
\[
\left(F^{[r]}\right)^{[q]}=F^{[r+q]},
\]
where $\left(F^{[r]}\right)^{[q]}$ denotes the shifted-increment
kernel formed from $F^{[r]}$ in the same way. 
\end{cor}

\begin{proof}
By \prettyref{thm:4-2}, 
\[
w_{r+s}-w_{r}=S_{r}w_{s}
\]
for all $r,s\geq0$. Hence 
\[
\begin{aligned}\Theta^{[r]}_{s,t} & =\left(w_{r+s}-w_{r}\right)\left(w_{r+t}-w_{r}\right)^{*}\\
 & =S_{r}w_{s}w^{*}_{t}S_{r}\\
 & =S_{r}\Theta_{s,t}S_{r}.
\end{aligned}
\]
Similarly, 
\[
\begin{aligned}F^{[r]}\left(s,t\right) & =\left(w_{r+s}-w_{r}\right)^{*}\left(w_{r+t}-w_{r}\right)\\
 & =\left(S_{r}w_{s}\right)^{*}S_{r}w_{t}\\
 & =w^{*}_{s}S^{2}_{r}w_{t}.
\end{aligned}
\]

Since $\left(S_{r}\right)_{r\geq0}$ is a semigroup of commuting positive
contractions, 
\[
S_{2r}-S_{2r+2q}=S_{2r}\left(I_{H_{F}}-S_{2q}\right)\geq0.
\]
Therefore 
\[
F^{[r]}\left(s,t\right)-F^{[r+q]}\left(s,t\right)=w^{*}_{s}\left(S_{2r}-S_{2r+2q}\right)w_{t}
\]
is a positive kernel. Taking $r=0$ gives $F^{[0]}=F$, and hence
\[
0\leq F^{[r+q]}\leq F^{[r]}\leq F.
\]

Finally, 
\[
\begin{aligned}\left(F^{[r]}\right)^{[q]}\left(s,t\right) & =F^{[r]}\left(q+s,q+t\right)-F^{[r]}\left(q+s,q\right)\\
 & \quad-F^{[r]}\left(q,q+t\right)+F^{[r]}\left(q,q\right)\\
 & =\left(w_{q+s}-w_{q}\right)^{*}S^{2}_{r}\left(w_{q+t}-w_{q}\right)\\
 & =\left(S_{q}w_{s}\right)^{*}S^{2}_{r}S_{q}w_{t}\\
 & =w^{*}_{s}S_{q}S^{2}_{r}S_{q}w_{t}\\
 & =w^{*}_{s}S_{2r+2q}w_{t}\\
 & =F^{[r+q]}\left(s,t\right).
\end{aligned}
\]
\end{proof}

Thus the two covariance kernels from \prettyref{sec:3} have tangent
counterparts: $F\left(s,t\right)=w^{*}_{s}w_{t}$ on $K$, and $\Theta_{s,t}=w_{s}w^{*}_{t}$
on $H_{F}$. The semigroup $\left(S_{r}\right)_{r\geq0}$ transports
the hidden covariance and gives the monotone family of visible increment
kernels.

\bibliographystyle{amsalpha}
\bibliography{ref}

\end{document}